\documentclass{llncs}
\usepackage{amsmath,amssymb}
\usepackage{array,multirow,makecell}
\usepackage{graphicx}
\usepackage{amsmath}
\usepackage{amsfonts}
\usepackage{amssymb}
\usepackage{caption,enumerate}
\usepackage{multicol}
\usepackage{float}
\usepackage{chngpage,subcaption}
\usepackage{hyperref}
\usepackage{cleveref}
\usepackage[all]{xy}
\usepackage{color}

\newcommand{\bbQ}{\mathbb{Q}}

\newcommand{\abs}[1]{|#1|}
\newcommand{\restrict}{\upharpoonright}

\title{On the degrees of constructively immune sets}
\author{Samuel D.~Birns \\ Bj{\o}rn Kjos-Hanssen\thanks{This work was partially supported by a grant from the Simons Foundation (\#704836 to Bj\o rn Kjos-Hanssen).}}
\author{Samuel D. Birns\inst{1} \and Bj{\o}rn Kjos-Hanssen \inst{1}\orcidID{0000-0002-1825-0097}\thanks{This work was partially supported by a grant from the Simons Foundation (\#704836 to Bj\o rn Kjos-Hanssen).}
}

\institute{
University of Hawai\textquoteleft i at M\=anoa, Honolulu HI 96822, U.S.A.,\\
\email{sbirns@hawaii.edu, bjoern.kjos-hanssen@hawaii.edu},\\ WWW home page:
\texttt{http://math.hawaii.edu/wordpress/bjoern/}
}

\begin{document}

\maketitle

\begin{abstract}
Xiang Li (1983) introduced what are now called constructively immune sets as an effective version of immunity.
Such have been studied in relation to randomness and minimal indices, and we add another application area: numberings of the rationals.
We also investigate the Turing degrees of constructively immune sets and the closely related $\Sigma^0_1$-dense sets of Ferbus-Zanda and Grigorieff (2008).

\end{abstract}

\section{Introduction}

Effectively immune sets, introduced by Smullyan in 1964 \cite{MR180485}, are well-known in computability as one of the incarnations of diagonal non-computability, first made famous by Arslanov's completeness criterion.
A set $A\subseteq\omega$ is \emph{effectively immune} if there is a computable function $h$ such that $\abs{W_e}\le h(e)$ whenever $W_e\subseteq A$, where $\{W_e\}_{e\in\omega}$ is a standard enumeration of the computably enumerable (c.e.) sets.

There is a more obvious effectivization of immunity (the lack of infinite computable subsets), however:
\emph{constructive immunity}, introduced by Xiang Li \cite{MR723334} who actually (and inconveniently) called it ``effective immunity''.

\begin{definition}
	A set $A$ is \emph{constructively immune}  if
	there exists a partial recursive $\psi$ such that for all $x$,
	if $W_x$ is infinite then $\psi(x)\downarrow$ and $\psi(x)\in W_x\setminus A$.
\end{definition}

The Turing degrees of constructively immune sets and the related $\Sigma^0_1$-dense sets have not been considered before in the literature, except that Xiang Li implicitly showed that they include all c.e.~degrees.
We prove in \Cref{sec:prevalent} that the Turing degrees of $\Sigma^0_1$-dense sets include all non-$\Delta^0_2$ degrees, all high degrees, and all c.e.~degrees.
We do not know whether they include \emph{all} Turing degrees.

The history of the study of constructive immunity seems to be easily summarized. After Xiang Li's 1983 paper,
Odifreddi's 1989 textbook \cite{MR982269} included Li's results as exercises, and Calude's 1994 monograph \cite{MR1323429} showed that the set
$RAND^C_t=\{x: C(x)\ge \abs{x}-t\}$ is constructively immune, where $C$ is Kolmogorov complexity. Schafer 1997 \cite{MR1654312} further developed an example involving minimal indices,
and Brattka 2002 \cite{MR2059846} gave one example in a more general setting than Cantor space. Finally in 2008 Ferbus-Zanda and Grigorieff proved an equivalence with constructive $\Sigma^0_1$-density.

\begin{definition}[Ferbus-Zanda and Grigorieff \cite{ferbuszanda2008refinment}]\label{def_eff_susc}
	A set $A\subseteq\omega$ is \emph{$\Sigma^0_1$-dense} if for every infinite c.e. set $C$, there exists an infinite c.e. set $D$ such that $D \subseteq C$ and $D \subseteq A$.

	If there is a computable function $f:\omega\to\omega$ such that for each $W_e$, $W_{f(e)}\subseteq A\cap W_e$, and $W_{f(e)}$ is infinite if $W_e$ is infinite, then $A$ is \emph{constructively $\Sigma^0_1$-dense}.
\end{definition}

We should note that while the various flavors of immune sets are always infinite by definition, Ferbus-Zanda and Grigorieff do not require $\Sigma^0_1$-dense sets to be co-infinite.

The $\Sigma^0_1$-dense sets form a natural $\Pi^0_4$ class in $2^\omega$ that coincides with the simple sets on $\Delta^0_2$ but is prevalent (in fact exists in every Turing degree) outside of $\Delta^0_2$ by \Cref{prevalent} below.

\section{$\Sigma^0_1$-density}

To show that there exists a set that is $\Sigma^0_1$-dense, but not constructively so, we use Mathias forcing.
A detailed treatment of the computability theory of Mathias forcing can be found in \cite{MR3210076}.
\begin{definition}
	A \emph{Mathias condition} is a pair $(d, E)$ where $d, E \subseteq \omega$, $d$ is a finite set, $E$ is an infinite computable set, and $\max(d) < \min(E)$.
	A condition $(d_2, E_2)$ \emph{extends} a condition $(d_1, E_1)$ if
	\begin{itemize}
		\item $d_1=d_2\cap(\max d_1+1)$, i.e., $d_1$ is an initial segment of $d_2$,
		\item $E_2$ is a subset of $E_1$, and
		\item $d_2$ is contained in $d_1 \cup E_1$.
	\end{itemize}
	A set $A$ is \emph{Mathias generic} if it is generic for Mathias forcing.
\end{definition}

\begin{theorem}\label{mathias}
If $A$ is Mathias generic, then
\begin{enumerate}
	\item $\omega\setminus A$ is $\Sigma^0_1$-dense.
	\item $\omega\setminus A$ is not constructively $\Sigma^0_1$-dense.
\end{enumerate}
\end{theorem}
\begin{proof}
\noindent 1. Let $W_e$ be an infinite c.e.~set. Let $(d,E)$ be a Mathias condition.

\indent Case (i): $E\cap W_e$ is finite. Then for any Mathias generic $A$ extending the condition $(d,E)$, $\omega\setminus A$ contains an infinite subset of $W_e$, in fact a set of the form $W_e\setminus F$ where $F$ is finite.

\indent Case (ii): $E\cap W_e$ is infinite. Then $E\cap W_e$ is c.e., hence has an infinite computable subset $D$.
Write $D=D_1\cup D_2$ where $D_1, D_2$ are disjoint infinite c.e.~sets. The condition $(d,D_1)$ extends $(d,E)$ and forces a Mathias generic $A$ extending it to be such that $\omega\setminus A$ has an infinite subset in common with $W_e$, namely $D_2$.

We have shown that for each infinite c.e.~set $W_e$, each Mathias condition has an extension forcing the statement that a Matias generic $A$ satisfies
\[
	\text{$\omega\setminus A$ has an infinite c.e.~subset in common with $W_e$.}\tag{*}
\]
Thus by standard forcing theory it follows that each Mathias generic satisfies ($*$).

\noindent 2. Let $f$ be a computable function.
It suffices to show that for each Mathias generic $A$, there exists an $i$ such that $W_i$ is infinite and $W_{f(i)}$ is either finite, or not a subset of $W_i$, or not a subset of $\overline A$. For this, as in (1) above it suffices to show that for each condition $(d,D)$ there exists a condition $(d',E')$ extending $(d,E)$ and an $i$ such that $W_i$ is infinite and $W_{f(i)}$ is either finite, or not a subset of $W_i$, or not a subset of $\overline A$ for any $A$ extending $(d',D')$.

Let $(d,E)$ be a Mathias condition and write $D=W_i$. If $W_{f(i)}$ is finite or not a subset of $W_i$ then we are done.
Otherwise there exists a condition $(d',E')$ extending $(d,E)$ such that $E'\cap W_{f(i)}$ is nonempty.
This can be done by a finite extension (making only finitely many changes to the condition).
\end{proof}

\begin{theorem}[{\cite[Proposition 3.3]{ferbuszanda2008refinment}}]\label{p:immune}
	A set $Z\subseteq\omega$ is constructively immune
	if and only if it is infinite and $\omega\setminus Z$ is
	constructively $\Sigma^0_1$-dense.
\end{theorem}
Since Ferbus-Zanda and Grigorieff's paper has not gone through peer review, we provide the proof.
\begin{proof}
	$\Leftarrow$: Let the function $g$ witness that $\omega\setminus Z$ is constructively $\Sigma^0_1$-dense.
	Define a partial recursive function $\varphi$ by stipulating that $\varphi(i)$ is the first number in the enumeration of $W_{g(i)}$, if any.

	$\Rightarrow$:
	Define a partial recursive function $\mu(i,n)$ by
	\begin{itemize}
		\item $\mu(i,0)=\varphi(i)$;
		\item $\mu(i,n+1)=\varphi(i_n)$, where $i_n$ is such that
		  $W_{i_n}=W_i\setminus\{\mu(i,m) : m\leq n\}$.
	\end{itemize}
	Let $g$ be total
	recursive so that $W_{g(i)}=\{\mu(i,m) : m\in\omega\}$.
	If $W_i$ is infinite then all $\mu(i,m)$'s are defined
	and distinct and belong to $W_i\cap Z$.
	Thus, $W_{g(i)}$ is an infinite subset of
	$W_i\cap Z$.
\end{proof}

Recall that a c.e.~set is \emph{simple} if it is co-immune.

\begin{theorem}[Xiang Li \cite{MR723334}]\label{xiang}
	Let $A$ be a set and let $\{\phi_x\}_{x\in\omega}$ be a standard enumeration of the partial computable functions.
	\begin{enumerate}
		\item\label{xiang-1} If $A$ is constructively immune then $A$ is immune and $\overline A$ is not immune.
		\item\label{xiang-2} If $A$ is simple then $\overline A$ is constructively immune.
		\item\label{xiang-3} $\{x: (\forall y)(\phi_x=\phi_y \to x \le y)\}$ is constructively immune.
	\end{enumerate}
\end{theorem}

\subsection{Numberings}
A \emph{numbering} of a countable set $\mathcal A$ is an onto function $\nu:\omega\to\mathcal A$.
The theory of numberings has a long history \cite{MR1720731}.
Numberings of the set of rational numbers $\mathbb Q$ provide an application area for $\Sigma^0_1$-density.
Rosenstein \cite[Section 16.2: Looking at $\mathbb Q$ effectively]{MR662564} discusses computable dense subsets of $\mathbb Q$.
Here we are mainly concerned with noncomputable sets.

\begin{proposition}\label{co-immune}
	Let $A \subseteq \omega$. The following are equivalent:
	\begin{enumerate}
		\item $\nu(A)$ is dense for every injective computable numbering $\nu$ of $\bbQ$;
		\item $A$ is co-immune.
	\end{enumerate}
\end{proposition}
\begin{proof}
	\indent (1)$\implies$(2): We prove the contrapositive.
	Suppose $\overline{A}$ contains an infinite c.e. set $W_e$. Consider a computable numbering $\nu$ that maps $W_e$ onto $[0,1]\cap\mathbb Q$. Then $\nu(A)$ is disjoint from $[0,1]$
	and hence not dense.

	\indent (2)$\implies$(1): We again prove the contrapositive.
	Assume that $\nu(A)$ is not dense for a certain computable $\nu$.
	Let $\{x_n:n\in\omega\}$ be a converging infinite sequence of rationals disjoint from $\nu(A)$.
	Then $\{\nu^{-1}(x_n):n\in\omega\}$ is an infinite c.e.~subset of $\overline{A}$.
\end{proof}

\begin{definition}
	A subset $A$ of $\bbQ$ is \emph{co-nowhere dense} if for each interval $[a,b] \subseteq \bbQ$, $[a', b'] \subseteq A$ for some $[a', b'] \subseteq [a,b]$.
\end{definition} 

\begin{proposition}\label{co-finite}
	A set is co-nowhere dense under every numbering iff it is co-finite. 
\end{proposition}
\begin{proof}
	Only the forward direction needs to be proven; the other direction is immediate.
	Let $A$ be a co-infinite set, and define $\nu$ by letting $\nu$ map $\omega \setminus A$ onto $[0,1]$. Then $A$ is not co-nowhere dense. 
\end{proof}

\begin{proposition}\label{lax-non-immune}
$A$ is infinite and non-immune iff there exists a computable numbering with respect to which $A$ is co-nowhere dense.
\end{proposition} 
\begin{proof}
Let $A$ be infinite and not immune. Thus, there is an infinite $W_e \subseteq A$ for some $e$.
Let $\nu$ be a computable numbering that maps $W_e$ onto $\bbQ \setminus \omega$. Then $A$ is co-nowhere dense under $\nu$.

Conversely, let $A$ be co-nowhere dense under some computable numbering $\nu$. Then $\nu^{-1}([a,b])$ is an infinite c.e. subset of $A$ for some suitable $a,b$.
\end{proof}

A set $D\subseteq\mathbb Q$ is \emph{effectively dense} if there is a computable function $f(a,b)$ giving an element of $D\cap (a,b)$ for $a<b\in \mathbb Q$.

\begin{proposition}\label{thm:numbering}
	A set $A$ is constructively $\Sigma^0_1$-dense iff it is effectively dense for all computable numberings.
\end{proposition}
\begin{proof}
	By \Cref{p:immune}, $A$ is constructively $\Sigma^0_1$-dense iff it is infinite and $\omega\setminus A$ is constructively immune.
	Constructive immunity of $\omega\setminus A$ implies effective density of $A$ since the witnessing function for constructive immunity can be be used to witness effective density.
	For the converse we exploit the assumption that we get to choose a suitable $\nu$.
\end{proof}

Let $A$ and $B$ be sets, with $B$ computable.
We say that $A$ is \emph{co-immune within $B$} if there is no infinite computable subset of $A^c\cap B$.
The following diagram includes some claims not proved in the paper, whose proof (or disproof) may be considered enjoyable exercises.
The quantifiers $\exists \nu$, $\forall \nu$ range over computable numberings of $\mathbb Q$.
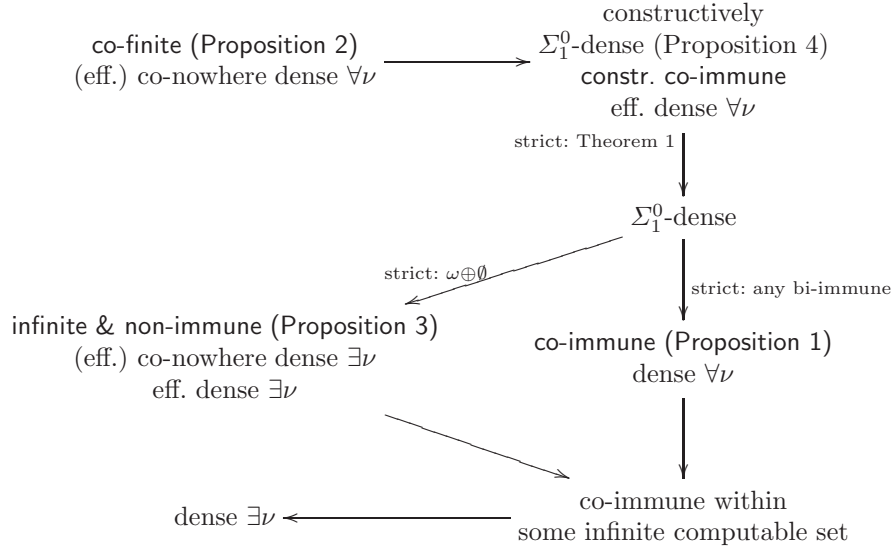
\[
	\xymatrix{
	\makecell{\textsf{co-finite (\Cref{co-finite})}\\ \text{(eff.) co-nowhere dense }\forall\nu}\ar[r] & \makecell{\text{constructively}\\ \text{$\Sigma^0_1$-dense (\Cref{thm:numbering})}\\ \textsf{constr.~co-immune}\\\text{eff.~dense }\forall\nu}\ar_{\text{strict: \Cref{mathias}}}[d] & \\
	 &\text{$\Sigma^0_1$-dense}\ar_{\text{strict: }\omega\oplus\emptyset}[dl]\ar^{\text{strict: any bi-immune}}[d]&
	 \\
	\makecell{\textsf{infinite \& non-immune (\Cref{lax-non-immune})}\\ \text{(eff.) co-nowhere dense }\exists\nu\\ \text{eff.~dense }\exists\nu \ar[dr]} &\makecell{\textsf{co-immune (\Cref{co-immune})}\\\text{dense }\forall\nu}\ar[d]& \\
	\text{dense }\exists\nu&\makecell{\text{co-immune within}\\\text{some infinite computable set}}\ar[l]
	}
\]

\section{Prevalence of $\Sigma^0_1$-density}\label{sec:prevalent}

In this section we investigate the existence of $\Sigma^0_1$-density in the Turing degrees at large.

\subsection{Closure properties and $\Sigma^0_1$-density}
\begin{proposition}\label{jan23}

	1. The intersection of two $\Sigma^0_1$-dense sets is $\Sigma^0_1$-dense.

	2. The intersection of two constructively $\Sigma^0_1$-dense sets is constructively $\Sigma^0_1$-dense.
\end{proposition}
\begin{proof}
	Let $A$ and $B$ be $\Sigma^0_1$-dense sets.
	Let $W_e$ be an infinite c.e.~set.
	Since $A$ is $\Sigma^0_1$-dense, there exists an infinite c.e.~set $W_d\subseteq A\cap W_e$. Since $B$ is $\Sigma^0_1$-dense, there exists an infinite c.e.~set $W_a\subseteq B \cap W_d$. Then $W_a\subseteq (A \cap B) \cap W_e$, as desired.
	This proves (1). To prove (2), let $f$ and $g$ witness the effective $\Sigma^0_1$-density of $A$ and $B$, respectively. Given $W_e$, we have $W_{f(e)}\subseteq A\cap W_e$ and then
	\[
	W_{g(f(e))}\subseteq B\cap W_{f(e)} \subseteq A\cap B\cap W_e.
	\]
	In other words, $g\circ f$ witnesses the effective $\Sigma^0_1$-density of $A\cap B$.
\end{proof}

\begin{corollary}
	Bi-$\Sigma^0_1$-dense sets do not exist. 
\end{corollary}
\begin{proof}
If $A$ and $A^c$ are both $\Sigma^0_1$-dense then by \Cref{jan23}, $A\cap A^c$ is $\Sigma^0_1$-dense, which is a contradiction.
\end{proof}

For sets $A$ and $B$, $A\subseteq^* B$ means that $A\setminus B$ is a finite set.

\begin{proposition}\label{upclosed}
	 \begin{enumerate}
		 \item\label{one} If $A$ is $\Sigma^0_1$-dense and $A \subseteq^* B$, then $B$ is $\Sigma^0_1$-dense.
		 \item\label{two} If $A$ is constructively $\Sigma^0_1$-dense and $A \subseteq^* B$, then $B$ is constructively $\Sigma^0_1$-dense. 
	\end{enumerate}
\end{proposition}
\begin{proof}
	Let $W_e$ be an infinite c.e.~set.
	Since $A$ is $\Sigma^0_1$-dense, there exists an infinite c.e.~set $W_d$ such that $W_d\subseteq A\cap W_e$.
	Let $W_c=W_d\setminus (A\setminus B)$.
	Since $A\setminus B$ is finite, $W_c$ is an infinite c.e.~set.
	Since $W_d\subseteq A$, we have $W_c = W_d \cap (B\cup A^c) = W_d \cap B$.
	Then, since $W_d\subseteq W_e$, we have
	\(
		W_c \subseteq B \cap W_e,
	\)
	and we conclude that $B$ is $\Sigma^0_1$-dense.
	This proves (1). To prove (2), if $f$ witnesses that $A$ is constructively $\Sigma^0_1$-dense then a function $g$ with $W_{g(e)}=W_{f(e)}\setminus (A\setminus B)$ witnesses that $B$ is constructively $\Sigma^0_1$-dense.
\end{proof}

\begin{proposition}\label{obvious}
	Let $B$ be a co-finite set. Then $B$ is constructively $\Sigma^0_1$-dense.
\end{proposition}
\begin{proof}
	The set $\omega$ is constructively $\Sigma^0_1$-dense as witnessed by the identity function $f(e)=e$.
	Thus by \Cref{two} of \Cref{upclosed}, $B$ is as well.
\end{proof}

As usual we write $A\oplus B = \{2x \mid x \in A\} \cup \{2x+1 \mid x \in B\}$.

\begin{proposition}\label{bjoernJan23}

1. If $X_0$ and $X_1$ are $\Sigma^0_1$-dense sets then so is $X_0\oplus X_1$.

2. If $X_0$ and $X_1$ are constructively $\Sigma^0_1$-dense sets then so is $X_0\oplus X_1$.
\end{proposition}
\begin{proof}
	Let $W_e=W_{c_0}\oplus W_{c_1}$ be an infinite c.e.~set.
	For $i=0,1$, since $X_i$ is $\Sigma^0_1$-dense there exists $W_{d_i}\subseteq X_i\cap W_{c_i}$ such that $W_{d_i}$ is infinite if $W_{c_i}$ is infinite.
	Then $W_{d_0}\oplus W_{d_1}$ is an infinite c.e.~subset of $(X_0\oplus X_1) \cap W_e$.

	This proves (1). To prove (2), if $d_i$ are now functions witnessing the effective $\Sigma^0_1$-density of $X_i$ then $W_{d_i(c_i)}\subseteq X_i\cap W_{c_i}$,
	and $W_{d_0(c_0)}\oplus W_{d_1(c_1)}$ is an infinite c.e.~subset of $(X_0\oplus X_1) \cap W_e$.
	Thus a function $g$ satisfying
	\[
		W_{g(e)}=W_{d_0(c_0)}\oplus W_{d_1(c_1)},
	\]
	where $W_e=W_{c_0}\oplus W_{c_1}$,
	witnesses the effective $\Sigma^0_1$-density of $X_0\oplus X_1$.
\end{proof}

\begin{theorem}\label{bjoern3:53}
	There is no $\Sigma^0_1$-dense set $A$ such that all $\Sigma^0_1$-dense sets $B$ satisfy $A\subseteq^* B$.
\end{theorem}
\begin{proof}
	Suppose there is such a set $A$. Let $W_d$ be an infinite computable subset of $A$.
	Let $G$ be a Mathias generic with $G\cap W_d^c=\emptyset$, i.e., $G\subseteq W_d$. Then $B:=G^c$ is $\Sigma^0_1$-dense by \Cref{mathias}.
	Thus $A\cap G^c$ is also $\Sigma^0_1$-dense by \Cref{jan23}.
	And $G \subseteq W_d \subseteq A$ and by assumption $A\subseteq^* G^c$ so we get $G\subseteq^* G^c$, a contradiction.
\end{proof}

These results show that the $\Sigma^0_1$-dense sets under $\subseteq^*$ form a non-principal filter whose Turing degrees form a join semi-lattice.

\begin{theorem}\label{thm:simple}
	Let $A$ be a c.e.~set.
	The following are equivalent:
	\begin{enumerate}
		\item $A$ is co-infinite and constructively $\Sigma^0_1$-dense.
		\item $A$ is co-infinite and $\Sigma^0_1$-dense.
		\item $A$ is co-immune.
	\end{enumerate}
\end{theorem}
\begin{proof}
$1\implies 2\implies 3$ is immediate from the definitions, and $3\implies 1$ is immediate from \Cref{p:immune} and \Cref{xiang}.
\end{proof}

\begin{theorem}
	Every c.e.~Turing degree contains a constructively $\Sigma^0_1$-dense set.
\end{theorem}
\begin{proof}
	Let $\mathbf a$ be a c.e.~degree.
	If $\mathbf a>\mathbf 0$ then $\mathbf a$ contains a simple set $A$, see, e.g., \cite{Soare}, so \Cref{thm:simple} finishes this case.
	The degree $\mathbf 0$ contains all the co-finite sets, which are constructively $\Sigma^0_1$-dense by \Cref{obvious}.
\end{proof}

\subsection{Cofinality in the Turing degrees of constructive $\Sigma^0_1$-density}
\begin{definition}\label{def_ve}
	For $k\ge 0$, let $I_k$ be intervals of length $k+2$ such that $\min(I_0)=0$ and $\max(I_k)+1=\min(I_{k+1})$.
	
	Let $V_e=\bigcup_{s\in\omega}V_{e,s}$ be a subset of $W_e$ defined by the condition that
	$x\in I_k$ enters $V_e$ at a stage $s$ where $x$ enters $W_e$ if
	this makes $\abs{V_{e,s}\cap I_k}\le 1$, and
	for all $j>k$, $V_{j,s} \cap I_k=\emptyset$.
\end{definition}

\begin{lemma}\label{dec31}
	There exists a c.e., co-infinite, constructively $\Sigma^0_1$-dense, and effectively co-immune set.
\end{lemma}
\begin{proof}
	Let $A=\bigcup_{e\in\omega} V_e$.
	$V_e$ is c.e. by construction, and if $W_e$ is infinite, $V_e$ is also infinite.
	So $V_e=W_{f(e)}$ is the set witnessing that $A$ is constructively $\Sigma^0_1$-dense.

	Moreover $A$ is coinfinite since $\abs{A\cap I_k}\le k+1 < k+2 = \abs{I_k}$ gives $I_k\not\subseteq A$ for each $k$ and
	\[
		\abs{\omega\setminus A} =\left| \left(\bigcup_{k\in\omega}I_k\right) \setminus A \right|
		= \left| \bigcup_{k\in\omega} (I_k\setminus A)\right|
		= \sum_{k\in\omega} \abs{I_k\setminus A} \ge \sum_{k\in\omega} 1 = \infty.
	\]
	The set $A$ is effectively co-immune because if $W_e$ is disjoint from $A$ then
	since as soon as a number in $I_k$ for $k\ge e$ enters $W_e$ then that number is put into $A$,
	$W_e \subseteq \bigcup_{k<e}I_k$ so $\abs{W_e}\le \sum_{k<e} (k+2) = \sum_{k\le e+1}k = \frac{(e+1)(e+2)}2$.
\end{proof}

\begin{theorem}
	For each set $R$ there exists a constructively $\Sigma^0_1$-dense, effectively co-immune set $S$ with $R\le_T S$.
\end{theorem}
\begin{proof}
	Let $R$ be any set, which we may assume is co-infinite.
	Let $A$ be as in the proof of \Cref{dec31}.
	Let $S\supseteq A$ be defined by
	\[
		S = A \cup \bigcup_{k\in R} I_k.
	\]
	Since $A\subseteq S$ and $S$ is co-infinite, $S$ is constructively $\Sigma^0_1$-dense and effectively co-immune.
	Since
	$k\in R\iff I_k\subseteq S$,
	we have $R\le_T S$.
\end{proof}

\subsection{Non-$\Delta^0_2$ degrees}

\begin{lemma}\label{jan26pigeon}
		Suppose that $T\subseteq 2^{<\omega}$ is a tree with only one infinite path.
		Then for each length $n$ there exists a length $m>n$ such that exactly one string of length $n$ has an extension of length $m$ in $T$.
\end{lemma}
\begin{proof}
Suppose not, i.e., there is a length $n$ such that for all $m>n$ there are at least two strings $\sigma_m,\tau_m$ of length $n$ with extensions of length $m$ in $T$. By the pigeonhole principle there is a pair $(\sigma,\tau)$ that is a choice of $(\sigma_m,\tau_m)$ for infinitely many $m$. Then by compactness both $\sigma$ and $\tau$ must be extendible to infinite paths of $T$.
\end{proof}

\begin{lemma}\label{jan26}
	Suppose that $T\subseteq 2^{<\omega}$ is a tree with only one infinite path $A$, and that $T$ is a c.e.~set of strings. Then $A$ is $\Delta^0_2$.
\end{lemma}
\begin{proof}
	By \Cref{jan26pigeon}, for each length $n$ there exists a length $m>n$ such that exactly one string of length $n$ has an extension of length $m$ in $T$.
	Using $0'$ as an oracle we can find that $m$ and define $A\upharpoonright n$ by looking for such a string.
	In fact, $T\le_T 0'$ and so its unique path $A\le_T 0'$ as well.
\end{proof}

\begin{theorem}\label{prevalent}
	Given $A \in 2^{\omega}$, let $\hat{A} := \left\{\sigma \in 2^{<\omega} \mid \sigma \prec A \right\}$ be the set of finite prefixes of $A$.
	If $A$ is not $\Delta^0_2$ then $\hat{A}$ is co-$\Sigma^0_1$-dense. 
\end{theorem}
\begin{proof}
	Let $A^*$ be the complement of $\hat A$.
	Let $W_e\subseteq 2^{<\omega}$ be an infinite c.e.~set of strings.
	Let $T$ be the set of all prefixes of elements of $W_e$.
	Then $T$ is an infinite tree, hence by compactness it has at least one infinite path.
	That is, there is at least one real $B$ such that all its prefixes are in $T$.

	\indent Case 1: The only such real is $B=A$. Then
	by \Cref{jan26}, $A$ is $\Delta^0_2$.

	\indent Case 2: There is a $B\ne A$ such that all its prefixes are in $T$.
	Let $\sigma$ be a prefix of $B$ that is not a prefix of $A$.
	Let $W_d=[\sigma]\cap W_e$.
	Since all prefixes of $B$ are prefixes of elements of $W_e$,
	there are infinitely many extensions of $\sigma$ that are prefixes of elements of $W_e$.
	Consequently $W_d$ is infinite.
	Thus, $W_d$ is our desired infinite subset of $A^*\cap W_e$.
\end{proof}

\subsection{High degrees}
\begin{definition}
	A set $A$ is \emph{co-r-cohesive} if its complement is r-cohesive. This means that for each computable (recursive) set $W_d$,
	either $W_d\subseteq^* A$ or $W_d^c\subseteq^* A$.
\end{definition}

\begin{definition}[{Odifreddi \cite[Exercise III.4.8]{MR982269}, Jockusch and Stephan \cite{MR1270396}}]
	A set $A$ is strongly hyperhyperimmune (s.h.h.i.) if for each computable $f:\omega\to\omega$ for which the sets $W_{f(e)}$ are disjoint,
	there is an $e$ with $W_{f(e)}\subseteq \omega\setminus A$.

	A set $A$ is strongly hyperimmune (s.h.i.) if for each computable $f:\omega\to\omega$ for which the sets $W_{f(e)}$ are disjoint and computable,
	with $\bigcup_{e\in\omega}W_{f(e)}$ also computable,
	there is an $e$ with $W_{f(e)}\subseteq \omega\setminus A$.
\end{definition}

\begin{proposition}\label{thm:shi}
	Every s.h.i.~set is co-$\Sigma^0_1$-dense.
\end{proposition}
\begin{proof}
	Let $A$ be s.h.i.
	Let $W_e$ be an infinite c.e.~set. Let $W_d$ be an infinite computable subset of $W_e$.
	Effectively decompose $W_d$ into infinitely many disjoint infinite computable sets,
	\[
		W_d = \bigcup_{i\in\omega} W_{g(d,i)}.
	\]
	For instance, if $W_d=\{a_0 < a_1 < \dots\}$ then we may let
	$W_{g(e,i)}=\{a_n: n=2^i(2k+1), i\ge 0, k\ge 0\}$.
	Since $A$ is s.h.i., there exists some $i_e$ such that $W_{g(d,i_e)}\subseteq A^c$.
	The sets $W_{g(d,i_e)}$ witness that $A^c$ is $\Sigma^0_1$-dense.
\end{proof}

Clearly r-cohesive implies s.h.i., and s.h.h.i. implies s.h.i.
It was shown by Jockusch and Stephan \cite[Corollary 2.4]{MR1270396} that the
cohesive degrees coincide with the r–cohesive degrees and (Corollary 3.10) that the s.h.i.~and s.h.h.i.~degrees coincide.

\begin{proposition}
	Every high degree contains a $\Sigma^0_1$-dense set.
\end{proposition}
\begin{proof}
	Let $\mathbf h$ be a Turing degree.
	If $\mathbf h\not\le \mathbf 0'$, then $\mathbf h$ contains a $\Sigma^0_1$-dense set by Theorem \ref{prevalent}.

	If $\mathbf h\le \mathbf 0'$ and $\mathbf h$ is high then since the strongly hyperhyperimmune and cohesive degrees coincide, and are exactly the high degrees \cite{MR360240},
	$\mathbf h$ contains a strongly hyperimmune set. Hence by Theorem \ref{thm:shi}, $\mathbf h$ contains a $\Sigma^0_1$-dense set.
\end{proof}

\subsection{Progressive approximations}
\begin{definition}
	Let $A$ be a $\Delta^0_2$ set.
	A computable approximation $\{\sigma_t\}_{t\in\omega}$ of $A$, where each $\sigma_t$ is a finite string and $\lim_{t\to\infty}\sigma_t=A$, is \emph{progressive} if
	for each $t$,
	\begin{itemize}
		\item if $\abs{\sigma_t}\le\abs{\sigma_{t-1}}$ then $\sigma_t\restrict(\abs{\sigma_t}-1) =  \sigma_{t-1}\restrict(\abs{\sigma_t}-1)$ (the last bit of $\sigma_t$ is the only difference with $\sigma_{t-1}$);
		\item if $\abs{\sigma_t} > \abs{\sigma_{t-1}}$ then $\sigma_{t-1}\prec\sigma_t$; and
		\item if $\sigma_t\not\prec\sigma_s$ for some $s>t$ then $\sigma_t\not\prec\sigma_{s'}$ for all $s'\ge s$ (once an approximation looks wrong, it never looks right again).
	\end{itemize}
	If $A$ has a progressive approximation then we say that $A$ is \emph{progressively approximable}.
\end{definition}
Note that a progressively approximable set must be $h$-c.e.~where $h(n)=2^n$.

\begin{theorem}\label{January 31}
	Let $A$ be a progressively approximable and noncomputable set. Let $\{\sigma_t\}_{t\in\omega}$ be a progressive approximation of $A$.
	Then $\{t: \sigma_t \prec A\}$ is constructively immune.
\end{theorem}
\begin{proof}
	Let $W_e$ be an infinite c.e.~set and let $T$ be an infinite computable subset of $W_e$.
	Since $A$ is noncomputable, we do not have $T\subseteq\{t:\sigma_t\prec A\}$.
	Since the approximation $\{\sigma_t\}_{t\in\omega}$ is progressive, once we observe a $t$ for which $\sigma_t\not\prec\sigma_s$, for some $s>t$, then we know that $\sigma_t\not\prec A$.
	Then we define $\varphi(e)=t$, and $\varphi$ witnesses that $\{t: \sigma_t \prec A\}$ is constructively immune.
\end{proof}

A direction for future work may be to find new Turing degrees of progressively approximable sets.

\bibliographystyle{plain}
\bibliography{cie21birns}

\end{document}